\documentclass{article}
\usepackage[utf8]{inputenc}
\usepackage{amsmath,amsthm,amssymb,amsfonts}
\usepackage{graphicx}
\usepackage{todonotes}
\usepackage{hyperref}
\usepackage{tikz} 
\usepackage{pgfplots}

\newtheorem{assum}{Assumption}
\newtheorem{thm}{Theorem}
\newtheorem{prop}[thm]{Proposition}
\newtheorem{cor}[thm]{Corollary}

\newcommand{\Z}{{\mathbb Z}}
\newcommand{\N}{{\mathbb N}}
\newcommand{\R}{{\mathbb R}}
\newcommand{\email}[1]{\href{mailto:#1}{#1}}

\title{Self-similarity in an exchangeable site-dynamics model} 
\begin{document}

\author{
Iddo Ben-Ari\footnote{Department of Mathematics, University of Connecticut, Storrs, CT 06269-1009, USA; 
\email{iddo.ben-ari@uconn.edu}}
\and
Rinaldo B. Schinazi \footnote{Department of Mathematics, University of Colorado, Colorado Springs, CO 80933-7150, USA;
\email{rinaldo.schinazi@uccs.edu}}
}

\maketitle

\begin{abstract}
We consider a model for which every site of $\N$ is assigned a fitness in $[0,1]$. At every discrete time all the sites are updated and each site samples a uniform on $[0,1]$, independently of everything else. At every discrete time  and independently of the past the environment is good with probability $p$ or bad with probability $1-p$. The fitness of each site is then updated to the maximum or the minimum between its present fitness and the sampled uniform, according to whether the environment is good or bad. Assuming the initial fitness distribution is exchangeable over the site indexing, the empirical fitness distribution is a probability-valued Markov process. We show that this Markov process converges to an explicitly-identified stationary distribution exhibiting  a self-similar structure.  %We define for each unit of time an empirical distribution for the fitness values present in the system. The exchangeable dynamics  with respect to the indexing of the sites makes the empirical distribution a probability measure-valued Markov chain. 

\end{abstract}

{\bf Keywords:}
Markov chain, exchangeable stochastic process, self-similar measure, population biology.

\section{Introduction and statement of results}
\subsection{Introduction}
We consider a discrete-time model on state space $[0,1]^\N$. At each time unit $t\in \Z_+$, each site $n\in\N$ has ``fitness'' $\eta_t (n) \in [0,1]$. The system evolves in time as follows. Let $p$ be a fixed number in $(0,1)$.  At any time $t\geq 0$ we generate a Bernoulli random variable $B_{t+1}$ with parameter $p$, and independently a sequence $(U_{t+1}(n):n\in \N)$ of IID uniform random variables on $[0,1]$. We update the model according to the following rules.
\begin{itemize}
   \item If $B_{t+1}=1$ then  for every $n\in \N$, $\eta_{t+1}(n)  = \max (\eta_t (n), U_{t+1}(n))$.
  \item If $B_{t+1}=0$ then  for every $n\in \N$, $\eta_{t+1}(n)  = \min (\eta_t (n), U_{t+1}(n))$.
\end{itemize}
We think of the $(B_t:t\in\N)$  as a time-evolving environment. At a given time the environment can be``good'' or ``bad'', it is the same for all sites.  The impact of the environment at a given site, however, depends on the current site fitness as well as its own ``luck'', independently of what happens to the other sites. 

The resulting process is a Markov chain. Additionally, the evolution of each individual site is also a Markov chain. Thus, the entire system can be viewed as a
system of infinitely many coupled Markov chains.

 The present model is related to the  so-called ``catastrophe'' models, see for instance \cite{Brockwell et al.} and  \cite{Brockwell}. In particular a model introduced in \cite{Neuts} and studied in \cite{Ben-Ari et al.} is reminiscent of our model. We now describe it.  At every discrete time the population increases by one unit with probability $p$ or is subjected to a catastrophe (i.e. it loses a random number of individuals) with probability $1-p$. Hence, similarly to the present model the evolution is caused by external environmental changes.  A major difference with the present model, however, is that catastrophe models only track the overall size of the population. We track every individual in the population.

Also related to the present model is the Bak-Sneppen model, see \cite{Bak93}. There, a fixed number of sites are arranged in a circle. At the initial time every site is given a fitness uniformly distributed in $[0,1]$. At every discrete time the site with the smallest fitness as well as its two nearest neighbors have their fitness updated by sampling three independent uniform random variables. One can think of Bak-Sneppen as a model for a group of species that evolve through internal competition alone. At the other extreme, our model follows a group of species that evolve through external  pressure alone. Our model can be compared to Bak-Sneppen through the marginal distribution of the stationary distribution at a site. For Bak-Sneppen the marginal distribution is believed to be uniform  (see \cite{ Meester2003} and \cite {Meester2004}) as it is for our model in the case $p=1/2$. For $p\not =1/2$ on the other hand  the marginal distribution for our model will be computed explicitly and shown not to be uniform.

Thanks to the exchangeable dynamics of our model de Finetti's Theorem provides an underlying Markov chain which turns out to be quite interesting in its own right. In particular, this underlying Markov chain has a stationary measure that has been studied in the fractals literature, see \cite{Strichartz95}. We believe this to be an interesting example of a self-similar measure that arises naturally from a rather simple Markov chain.

%In words, sites are simultaneously updated at every time $t\in \N$. If $B_{t+1}=1$ the fitness of every site goes up or stays put. If $B_{t+1}=0$ the fitness of every site goes down or stays put.

%Note that the time evolution of each site $(\eta_{t}(n):t\in\Z_+)$ is a Markov chain whose transition kernel does not depend on $n$. 
% We can think of this as a model for the evolution of a population living under the same, but varying environment. If the environment is ``good'' at time $t$ (i.e. $B_{t+1}=1$) then the fitness goes up or stays put for all sites at time $t+1$.   On the other hand if the environment is ``bad'' (i.e. $B_{t+1}=0$) then the fitness goes down or stays put for all sites at time $t+1$. How much each individual goes up or down depends on the present fitness of the individual and of its own ``luck'' independent of all others. %Note that the change in fitness for one individual is independent of what happens to the other individuals. The main purpose of this paper is to study the collective behavior of this population.
 
%The main purpose of this paper is to study the collective behavior of this population.
 
\subsection{Main Results} 
The dynamics for the system are exchangeable. That is,  if the initial distribution is exchangeable, for example IID, then $(\eta_t (\cdot):n\in\Z_+) $ is an exchangeable sequence for all $t$. We can apply de Finetti's theorem. In order to do so, assume that the initial distribution is exchangeable. Let $u\in [0,1]$ and let 
$$ {\bf I}_t (n,u) = {\bf 1}_{\{\eta_t (n) \le u\}}.$$ 

Then for every $t$ and $u$, the family $ (I_t(n,u):n\in\N)$ is an exchangeable sequence. In particular, it follows from de Finetti's theorem that there exists a random variable $\Theta_t (u)$, measurable with respect to the exchangeable $\sigma$-algebra ${\cal E}$ such that the distribution of  $(I_t(n,u):n\in\N)$, conditioned on ${\cal E}$ is IID with a Bernoulli distribution with parameter  $\Theta_t(u)$. Furthermore, 

$$ \Theta_t (u) = \lim_{N\to\infty} \frac{1}{N} \sum_{n=1}^N {\bf I}_t (n,u), \mbox{ a.s. }$$ 

A key observation is the following recursion formula for $ \Theta_t (u)$.
\begin{equation}
\label{eq:thetat_evol}
\Theta_{t+1}(u)  = \begin{cases} \Theta_t(u)  u &\mbox{ if } B_{t+1} = 1 \\ \Theta_t(u)  + (1-\Theta_t(u))u & \mbox{ if } B_{t+1}=0. \end{cases}
\end{equation}
We now prove this formula. Note that if $B_{t+1}=1$ then 
$${\bf I}_{t+1}(n,u)={\bf I}_t(n,u){\bf 1}_{\{U_{t+1}(n)\leq u\}}.$$
On the other hand if $B_{t+1}=0$ then 
$${\bf I}_{t+1}(n,u)={\bf I}_t(n,u)+(1-{\bf I}_t(n,u)){\bf 1}_{\{U_{t+1}(n)\leq u\}}.$$
The formula \eqref{eq:thetat_evol} now follows from the Law of Large Numbers for exchangeable sequences and the independence of $\eta_t$ and $(U_{t+1}(n):n\in \N)$.

Note that the function $u \to \Theta_t(u)$ is a random cumulative distribution function, and its distribution determines the distribution of ${\bf \eta}_t$, conditioned on ${\cal E}$. Furthermore, $t\to\Theta_t(\cdot)$ is a Markov chain on the space of CDFs. 

\begin{thm}
\label{thm:limit}
Let $G_0,G_1,\dots$ be $\mbox{IID}$ $\mbox{Geom}(1-p)$-distributed RVs, and for $k\in\Z_+$, let  $T_k= G_0 + \dots +G_k$.  Then the distribution of the random CDF  $\Theta_t(\cdot)$ converges as $t\to\infty$ to that of the random CDF $\Theta_\infty$ given by 
\begin{equation}
\Theta_\infty (u) = \label{eq:limit}  \sum_{k=0}^\infty u^{T_k} \left( \frac{1-u}{u}\right)^k,~ u \in (0,1).
\end{equation} 
\end{thm} 

Observe that since the CDF-valued process $(\Theta_t (\cdot):t\in\Z_+)$ is a Markov process,  the convergence in the theorem implies that the process possesses a unique stationary distribution given by the expression in \eqref{eq:limit}, though this can be also verified by a direct calculation.  For every $u$ in $(0,1)$ the probability distribution of $\Theta_\infty(u)$ belongs to a family of probability measures known in the literature as self-similar measures associated with an iterated function system. Fix $u \in (0,1)$ and define the function system  $\{S_0,S_1\}$, $S_j :[0,1]\to [0,1]$ by $S_0(x) = u+ (1-u)x$ and  $S_1 (x) = ux$. By \eqref{eq:thetat_evol} the unique stationary distribution $\mu_u$ for the process
$(\Theta_t (u):t\in\Z_+)$ satisfies
\begin{equation} 
\label{eq:self_similar} \mu_u = (1-p) \mu_u \circ S_0^{-1} + p \mu_u \circ S_1^{-1}.
\end{equation}
The cumulative distribution function $G_u$ corresponding to $\mu_u$ turns out to have remarkable properties, see Figure \ref{fig:1}. It is continuous, see Proposition \ref{prop:Gu} but singular with respect to the Lebesgue measure, see Proposition \ref{prop:Sing}. Moreover, we have an explicit formula for $G_u$ on a dense set of $[0,1]$, see equation (\ref{eq:Gu_J}).

\begin{figure}[ht] 
\includegraphics[scale=0.5]{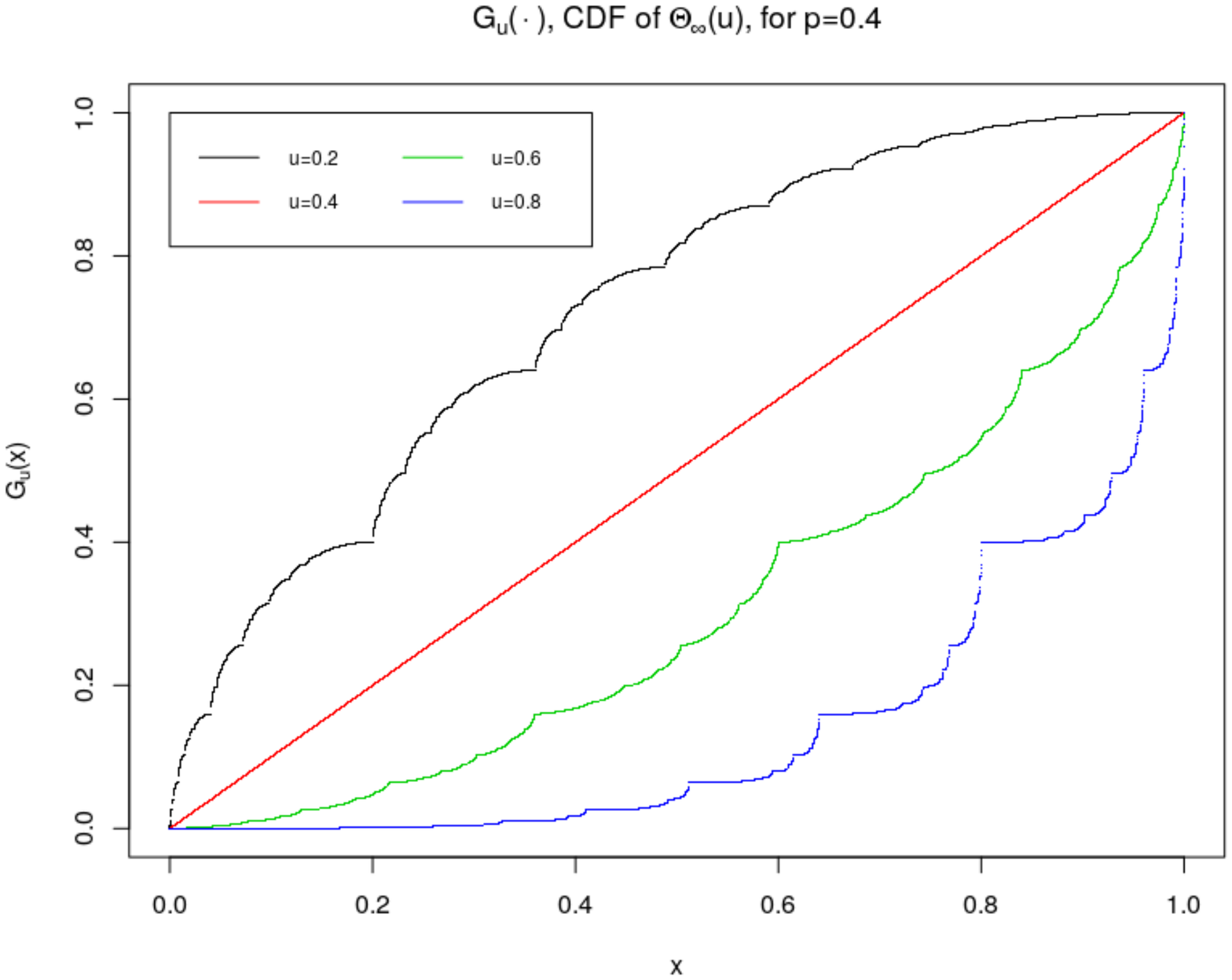}
\caption{The cumulative distribution function of $\Theta_\infty(u)$ for some values of $u$ and with $p=0.4$}
\label{fig:1}
\end{figure} 

The results above are related to a long standing open problem. Consider the following iterated function system. Let
$\{T_0,T_1\}$, $T_j :[0,1]\to [0,1]$ by $T_0(x) = 1-u+ ux$ and  $T_1 (x) = ux$. 
In \cite{Erdos} examples for $u$'s in $(1/2,1)$ are given for which the stationary distribution corresponding to the system $\{T_0,T_1\}$ is continuous but singular with respect to the Lebesgue measure. It is also known that the stationary distribution is absolutely continuous for some values in $(1/2,1)$ and singular for all values in $(0,1/2)$.  Unfortunately our method does not apply to the system $\{T_0,T_1\}$. As far as we know the question of determining for which $u$'s in $(1/2,1)$ the stationary measure is singular is still open, see also the discussion in \cite[p. 24]{Barnsley}. However, we prove the following representation for the unique stationary distribution of the system $\{T_0,T_1\}$. For a fixed $u$ in $(0,1)$, the stationary distribution has the same distribution as 
$$\frac{1-u}{u}\sum_{k\geq 0} u^{T_k},$$
where $T_k = G_0 +\dots + G_k$, and $G_0,G_1,\dots $ are IID $\mbox{Geom}(1-p)$. This representation formula is an application of Corollary \ref{cor:twofunctions}.

The exchangeability of the random variables $(\eta_t(n):n\in\N)$ implies that for every $t\in\Z_+$, $N\in\N$ and and every $(u_1,u_2,\dots,u_N)$ in $[0,1]^N$,
$$ P( \bigcap_{n=1}^N \{\eta_t(n)  \le u_n\}|{\cal E}) = \prod_{n=1}^N P(\eta_t(n) \le u_n|{\cal E}) = \prod_{n=1}^N \Theta_t (u_n)$$ 
As a consequence, we have the following 
\begin{cor} 
\label{cor:pi_eta} 
The Markov chain $(\eta_t:t\in\Z_+)$ has a unique stationary distribution $\pi_p$ given by 

$$ \pi_p \left(\eta\in [0,1]^\N:\bigcap_{n=1}^N\{0\leq \eta(n)\leq u_n\}\right) = E [ \prod_{n=1}^N \Theta_\infty(u_n) ],$$
for every natural number $N$ and every $(u_1,u_2,\dots,u_N)$ in $[0,1]^N$.
\end{cor}

The moments of $\Theta_\infty(u)$ can be computed iteratively using the following formula.

\begin{prop}
\label{prop:moments} 
For  $k=1,2,\dots$ 
$$E[ \Theta_\infty^k(u)] = \frac{1-p}{1-p u^k - (1-p) (1-u)^k} \sum_{j=0}^{k-1} \binom{k}{j} u^{k-j} (1-u)^{j} E[\Theta_\infty^{j}(u)].$$ 
\end{prop}

The marginal distribution at a single site can be computed explicitly,

\begin{cor}
\label{cor:single_site}
The single site chain (at any site) has a unique stationary distribution  with CDF $F_p$ and density  $f_p$ given by 
\begin{align*} F_p (u) &= \frac{(1-p)u}{(1-p)u + p(1-u)}\\
f_p(u) &= \frac{p}{1-p} (\frac{F_p(u)}{u})^2  = \frac{p(1-p)}{((1-2p)u^2 + p)^2}
\end{align*} 
\end{cor}

Note the following: 
\begin{equation}
\label{eq:symmetry} 
F_p(u) = F_{1-u}(1-p)
\end{equation} 
as well as 
$$ f_p(0+) = \frac{1-p}{p}=\frac{1}{f_p(1-)}.$$

Setting $k=1$ in Proposition \ref{prop:moments}  we obtain 
\begin{equation}
\label{eq:expectation} E[ \Theta_\infty (u) ]= \frac{(1-p)u} {1-p u-(1-p)(1-u)}= \frac{(1-p)u}{(1-p)u + p(1-u)}.
\end{equation} 
Using \eqref{eq:expectation} in Corollary \ref{cor:pi_eta} proves Corollary \ref{cor:single_site}.
\medskip

\subsection{Self-similarity} 
Fix $u \in (0,1)$ and define the function system  $\{S_0,S_1\}$, $S_j :[0,1]\to [0,1]$ by $S_0(x) = u+ (1-u)x$ and  $S_1 (x) = ux$. Then \eqref{eq:thetat_evol} can be written as 
\begin{equation}
\label{eq:theta_function_system} 
\Theta_{t+1}(u) = S_{B_{t+1}} (\Theta_t(u)) 
\end{equation} 
It therefore follows that a probability distribution $\mu_u$  on $[0,1]$ is stationary for $\Theta_{\cdot}(u)$ if and only if 
\begin{equation} 
\label{eq:self_similar} \mu_u = (1-p) \mu_u \circ S_0^{-1} + p \mu_u \circ S_1^{-1}.
\end{equation}
That is, $\mu_u$ is  self-similar with respect to the function system $\{S_0,S_1\}$ and the probability vector  $(1-p,p)$, see \cite{Strichartz95}. 

As the images of $S_0$ and $S_1$ are $[u,1]$ and $[0,u]$, it follows that  $\mu_u$ satisfies \eqref{eq:self_similar} if and only if its CDF $G_u$ satisfies 
%In terms of the CDF of $\mu_u$, $G_u$, \eqref{eq:self_similar} can be written as 
%$$ G_u (y) = \begin{cases} p G_u (y/u) & y \in (0,u) \\ p & y= u \\ p+(1-p) G_u ((y-u)/(1-u)) & y \in (u,1)\end{cases}
%$$ 
%This can be summarized as follows. For $z\in [0,1]$
\begin{align}
    \label{eq:scaling}  
 G_u (uz) &= p G_u (z)\\
 \label{eq:shifting}
 G_u (u+ (1-u) z) &= p + (1-p) G_u (z)
 \end{align}
 for all $z \in [0,1]$. 
Here is an explicit formula for the CDF of $\Theta_\infty(u)$ on a dense set.
\begin{prop} 
\label{prop:Gu}
Let $u \in (0,1)$. 
 Let $D$ be the set of numbers of the form 
 \begin{equation} 
 \label{eq:the_y}
y = \sum_{l=1}^m  u^{n_l}(1-u)^{l-1}
\end{equation} 
for some $m\in\N$ and for $n_1,n_2,\dots,n_m$ as follows, 
\begin{itemize} 
\item If $m=1$, then $n_1\in \Z_+\cup \{\infty\}$. 
\item If $m >1$,  then
\begin{itemize} 
\item $n_1,\dots,n_{m-1}\in \N \cup \{\infty\}$
satisfying  $n_1\le \dots \le n_{m-1}$;  and
\item $n_m \in \Z_+ \cup \{\infty\}$ satisfying $n_{m-1}\le n_{m}+1$. 
\end{itemize}
\end{itemize} 
 \begin{enumerate} 
\item The set $D$ is dense in $[0,1]$.
 \item Let $G_u$ denote the CDF of $\Theta_\infty(u)$. Then,  $G_u$ is continuous, and for $y$ defined by (\ref{eq:the_y}), 
 \begin{equation} 
 \label{eq:Gu_J} G_u (y) = \sum_{l=1}^m  p^{n_l}(1-p)^{l-1}.
 \end{equation} 
\end{enumerate} 
\end{prop}
It follows from the proposition that for $u=p$, we have  $G_p(y)=y$ for all $y$ in $D$. Since this is a dense set, it follows that $\mu_p$ is uniform on $[0,1]$. 

From \eqref{eq:scaling}  $G_u (u^n) = p^n$,  for all $n\in\N$. Then by  right continuity, $G_u(0)=0$. 
Note that \eqref{eq:Gu_J} is a consequence of \eqref{eq:scaling}  and
 \eqref{eq:shifting}. The rest of the proposition will be proved below.

\subsection{Proof of Proposition \ref{prop:moments}}

Since the distribution of $\Theta_\infty(u)$ satisfies \eqref{eq:self_similar}, it follows that for any bounded Borel function $f$, 
\begin{equation}
    \label{eq:self_similar_f}
    E[ f (\Theta_\infty(u) )] = p E[f (u \Theta_\infty(u)) ]+(1-p) E [ f (u+ (1-u) \Theta_\infty(u))].
\end{equation} 

 We can harness 
\eqref{eq:self_similar_f} to calculate the moments of $\Theta_\infty(u)$. Indeed, let $f(z) =e^{\lambda z}$, and let $\varphi_u(\lambda) = E [ f(\Theta_\infty(u))]$ be the moment generating function for $\Theta_\infty(u)$. Then since $\Theta_\infty(u)$ takes values in $[0,1]$, $\varphi_u$ is entire, and $\varphi_u^{(k)}(0)=E[\Theta_\infty(u)^k]$. From \eqref{eq:self_similar_f}, we have 

$$ \varphi_u(\lambda) = p \varphi_{u} ( u \lambda) + (1-p) e^{u \lambda} \varphi_u( (1-u) \lambda).$$ 
By taking derivatives with respect to $\lambda$ we obtain, 
$$ \varphi_u^{(k)}(\lambda) = pu^k \varphi_u^{(k)} ( u \lambda) + (1-p)\sum_{j=0}^k \binom{k}{j} u^{k-j}e^{\lambda u}  (1-u)^{j} \varphi_u^{(j)}((1-u) \lambda).$$ 

Letting $\lambda=0$ and doing the algebra, we obtain Proposition \ref{prop:moments}.

\subsection{Proof of Proposition \ref{prop:Gu}}
\begin{proof}[Proof of Proposition \ref{prop:Gu}]
The continuity of $G_u$ is proved below in Proposition \ref{prop:CDF}. We show now that $D$ is dense in $[0,1]$. 

Let $J_{1,d} = \{1,u,u^2, \dots , 0\}$, and let $J_{1,u} = \{u+ (1-u) x:x\in J_{1,d}\}$. Then  every element in $y\in J_{1,u}$ is of the form $u+(1-u) u^k$ for some $k\in \Z_+\cup \{\infty\}$ and by \eqref{eq:shifting}  for such $y$, $G_u (y) = p + (1-p) p^k$. 
Iterating the definition, for $m\geq 1$ let 

\begin{equation} 
\label{eq:iterationd} J_{m+1,d} = \{u^k x: x \in J_{m,u},k \in \Z_+\cup\{\infty\}
\end{equation} 
\begin{equation} 
\label{eq:iterationu}
J_{m+1,u} = \{u+(1-u)x:x\in J_{m+1,d}\}
\end{equation}

We will now prove that for every $m\in \N$,  $J_{m,d}$ is the set of numbers expressible as \eqref{eq:the_y}. The case $m=1$ follows directly from the definition of $J_{1,d}$. We continue to the general case and will apply induction, with the base case being $m=2$. From \eqref{eq:iterationd}, we have 
$$ J_{2,d} =   \{u^{1+k_2}+ (1-u)u^{k_1+k_2}:k_1,k_2 \in \Z_+\cup\{\infty\}\}.$$ 
Therefore letting $n_1 = 1+k_2\in \N \cup\{\infty\}$ and $n_2=k_1+k_2\in \Z_+$, we also have $n_1 \le n_2 +1$, and \eqref{eq:the_y} holds for $m=2$.

We turn to the induction step. From the induction hypothesis on $J_{m,d}$ and the definition of $J_{m,u}$, we have that  $y\in J_{m,u}$  if and only if 
\begin{equation}
\label{eq:Jm_u} y = u+ (1-u) \sum_{l=1}^m u^{n_l}(1-u)^{l-1}.
\end{equation} 
and then from \eqref{eq:iterationd},  $y\in J_{m+1,d}$ if and only if there exists $k_{m+1}\in \Z_+\cup \{\infty\}$ such that 

\begin{align*} 
\label{eq:Jm1_d}
y&= u^{k_{m+1}+1} +\sum_{l=1}^{m} u^{n_l+k_{m+1}} (1-u)^{l}\\
& = u^{k_{m+1}+1} + u^{n_1+k_{m+1}}(1-u) + u^{n_2+k_{m+1}}(1-u)^2 + \dots + u^{n_m+k_{m+1}}(1-u)^{m}\\
& =\sum_{l=1}^{m+1} u^{n'_{l}}(1-u)^{l-1},
\end{align*}
where $n'_1=k_{m+1}+1$, and for $l=2,\dots,m+1$, $n'_{l} = n_{l-1}+k_{m+1}$. By the induction hypothesis, 
\begin{itemize} 
\item $1\le n_1\le \dots \le  n_{m-1}$, and therefore it follows that $1\le n'_1\le \dots \le n'_m$.
\item $n_{m-1}\le n_m +1$, and therefore 
$$n'_m = n_{m-1} +k_{m+1}\le n_m +1 + k_{m+1} = n'_{m+1}+1.$$
\end{itemize} 
This completes the proof that for all $m\geq 1$, $J_{m,d}$ is the set of numbers expressible as \eqref{eq:the_y}.

Finally, it remains to show that the union of 
 $J_{m,d}$, $m\in\Z_+$, is dense in $[0,1]$. Let
 
$$R_m = \sum_{j=1}^m (1-u)^{j-1}.$$ 
Then $R_\infty =\lim_{m\to\infty} R_m = \frac{1}{u}$.

Let $x \in (0,1)$, and let $n_1 = \min \{n: u^n \le x\}$. Then necessarily, $n_1 \in \N$, and $x \in [u^{n_1},u^{n_1-1})$. Since 
$$u^{n_1} = u^{n_1}R_1 < u^{n_1}R_2 < \dots < u^{n_1} R_\infty = u^{n_1-1},$$
there exists $m_1 \in \N$ such that 
$$u^{n_1}R_{m_1}\le x < u^{n_1}R_{m_1+1}.$$ 

Set $x_1 = u^{n_1}R_{m_1}$, and let $x_2  = x - x_1$. Then 
$x_2 \in u^{n_1}[0,(1-u)^{m_1})$.  Therefore we can write 
$x_2 = u^{n_1}(1-u)^{m_1} y_1$ where $y_1 \in [0,1)$. In other words, 
$$ x = x_1 + x_2 = \underset{\in J_{m_1,d}}{\underbrace{u^{n_1}R_{m_1}}} + u^{n_1}(1-u)^{m_1}y_1,$$ 
where $y_1 \in [0,1)$. If $y_1=0$, we stop. Otherwise, we  iterate the process for $y_1$. That is, we find natural numbers $n_2$ and $m_2$ such that
$$y_1=u^{n_2}R_{m_2}+u^{n_2}(1-u)^{m_2}y_2$$
where $y_2\in [0,1)$. Hence,
$$x=u^{n_1}R_{m_1}+u^{n_1+n_2}(1-u)^{m_1}R_{m_2}+u^{n_1+n_2}(1-u)^{m_1+m_2}y_2.$$
It is easy to check that $u^{n_1}R_{m_1}+u^{n_1+n_2}(1-u)^{m_1}R_{m_2}$ is in $J_{m_1+m_2,d}$. If $y_2=0$ we stop. If not we continue this process to get arbitrarily close to $x$. 
\end{proof} 
\section{More general Self-Affine Markov Chains}
In this section we describe dynamics that generalize the Markov chain dynamics given in \eqref{eq:thetat_evol} in a natural way, derive the corresponding limit results, and use this to prove Theorem \ref{thm:limit}. 
\subsection{Setup and Convergence of Marginals} 
 We begin with the setup. 

Fix $K \in\N$. Let $(p_0,\dots,p_K)$ be a probability vector with strictly positive entries. That is 

\begin{equation} 
\label{eq:prob_vector} 
\min_{0\leq j\leq K}  p_j > 0,\quad  \sum_{j=0}^
K p_j =1
\end{equation} 
Also, let 
\begin{equation} 
\label{eq:coefficients}
\begin{cases} 
0\le a_j(u)\le a_j (u) + b_j(u)\le 1& ~j =0,\dots,K,u \in [0,1) \\
j \to (a_j (u), b_j(u)) & \mbox{ is }1-1\mbox{ for every }u \in [0,1). 
\end{cases}
\end{equation} 
 For each $u$ and $j =0,\dots,K$, let $S_j (u) x = a_j (u) + b_j (u) x$ be an affine map from  $[0,1)$ into $[0,1]$. 

Next, let 
\begin{equation}
\label{eq:IID}
 (B_t:t\in \N)\mbox{ be IID with } P(B_1=j) = p_j,~j=0,\dots,K.
\end{equation}

We now define a family of Markov chains indexed by  $u\in [0,1)$ according to the rule: 

\begin{equation} 
\label{eq:master_chain} 
\begin{cases} 
\Theta_0(\cdot) \in [0,1]\\
\Theta_t(u)  = S_{B_{t}}(u) \Theta_{t-1}(u)= a_{B_t}(u) + b_{B_t}(u) \Theta_{t-1}(u).
\end{cases}
\end{equation} 

\begin{assum}
\label{assum:theta}
Let $K\in \N$. Assume that \eqref{eq:prob_vector}, \eqref{eq:coefficients}, \eqref{eq:IID} hold, and let  $(\Theta_t(\cdot):t\in\Z_+)$ be as in \eqref{eq:master_chain}. 
\end{assum} 
Finally, define 
$$N_{t,j} = |\{s\le t: B_s = j\}|,~j\in 0,\dots,K.$$ 
We have the following 
\begin{thm}\label{thm:convergence Theta}
Let Assumption \ref{assum:theta} hold.  Then for every $u \in [0,1)$ 
\begin{enumerate} 
\item  the process $(\tilde \Theta_t (u):t \in\Z_+)$ defined as 
$$ 
 \tilde \Theta_t(u)  = \sum_{s=1}^t a_{B_s}(u)  \prod_{i=0}^K b_i^{N_{s-1,i}}(u)+ \Theta_0(u)\prod_{i=0}^K b_i^{N_{t,i}}(u).
$$ 
is identically distributed as $(\Theta_t (u):t\in\Z_+)$. 
\item 
 \begin{equation} 
\label{eq:thetainfinity_u} \lim_{t\to\infty} \tilde \Theta_t(u) =\tilde \Theta_\infty(u) =  \sum_{s=1}^\infty a_{B_s}(u) \prod_{i=0}^K b_i^{N_{s-1,i}}(u),\mbox{ a.s.} 
\end{equation} 

\item $$ E| \tilde \Theta_t(u) - \tilde \Theta_\infty(u) | \le  \rho^t(u) \frac {2-\rho(u)}{1-\rho(u)},$$
where $\rho(u) = \sum_{i=0}^K p_i b_i(u)\in (0,1)$. 
\end{enumerate} 
\end{thm}
\begin{proof} 
From \eqref{eq:master_chain} we see that  $\Theta_t(u)$ is a deterministic function of $\Theta_0, B_0,\dots,B_t$ and of $a_0(u),\dots,a_K(u),b_0(u),\dots,b_K(u)$. In order to keep the notation simple, in what follows we fix $u$, and  suppress the dependence on it. 

By \eqref{eq:master_chain} we get
\begin{align*} 
\Theta_t  &=  a_{B_t}  + b_{B_t} a_{B_{t-1}} + b_{B_t}b_{B_{t-1}}a_{B_{t-2}} \\
 & \quad\quad + \dots +  b_{B_t}b_{B_{t-1}} \cdots b_{B_2} a_{B_1} +  b_{B_t}b_{B_{t-1}} \cdots b_{B_1} \Theta_0.\\
& = \sum_{s=1}^t a_{B_s} \prod_{s<k\le t} b_{B_k}+ \Theta_0\prod_{k=1}^t  b_{B_k}.
\end{align*} 
Now fix $t\in \N$. For $r=1,\dots,t$, let  $\tilde B_r = B_{t-r+1}$, so $\tilde B_1 = B_t, \tilde B_2 = B_{t-1},\dots$. Also, let   $\tilde N_{r,i}=|\{1\le \rho \le r: \tilde B_\rho =i\}|$.  Then, 
$$N_{t,i}-N_{s,i} ={\bf 1}_{\{B_{s+1}=i\}} + \dots + {\bf 1}_{\{B_t=i\}} = {\bf 1}_{\{\tilde B_1=i\}} + \dots + {\bf 1}_{\{\tilde B_{t-s}=i\}}=\tilde N_{t-s,i},$$ and $B_s = \tilde B_{t-s+1}$. With this, we can  write 

\begin{align*} a_{B_s} \prod_{s< k \le t} b_{B_k} &=a_{B_s}\prod_{i=0}^K b_i^{N_{t,i}- N_{s,i}}= a_{\tilde B_{t-s+1}}\prod_{i=0}^K b_i^{\tilde N_{t-s,i}}, 
\end{align*} 
and so changing the summation from $s$ to $r=t-s+1$, we obtain 
$$\Theta_t = \sum_{r=1}^{t} a_{\tilde B_{r}} \prod_{i=0}^K b_i^{\tilde N_{r-1,i}}  + \Theta_0  \prod_{i=0}^K b_i^{\tilde N_{t,i}}.$$
Since the joint distribution of $\tilde B_1,\tilde B_2,\dots \tilde B_t$ coincides with that of $B_1,\dots,B_t$, it follows that 
\begin{equation}
    \label{eq:partial_sum} 
\Theta_t  \overset{\mbox{\scriptsize dist}}{=}  \tilde \Theta_t := \sum_{s=1}^t a_{B_s} \prod_{i=0}^K b_i^{N_{s-1,i}}+ \Theta_0\prod_{i=0}^K b_i^{N_{t,i}}.
\end{equation} 
Note that this equality in distribution holds also for the function $\Theta_t (u)$. Note that the expression on the right hand side is a partial sum of a series. By \eqref{eq:prob_vector}  $\max_{j} p_j<1$, and by \eqref{eq:coefficients}, $\max_j b_j\le 1$. In particular, $\rho=\rho(u)  = \sum_{i=0}^K p_i b_i\le 1$ with equality if and only if $b_j=1$ for all $j$, which violates our assumptions. 
Using the formula for the probability generating function of a multinomial distribution,
$$ E\left[ \prod_{i=0}^K b_i^{N_{s-1},i} \right] = \left(\sum_{i=0}^K b_i p_i\right)^{s-1}=\rho^{s-1},$$ 
monotone convergence guarantees that the partial sum in  \eqref{eq:partial_sum} converges a.s. to $\tilde \Theta_\infty$ which is defined by \eqref{eq:thetainfinity_u}. We also have 

$$ \tilde \Theta_\infty - \tilde \Theta_t = \sum_{s=t+1}^\infty a_{B_s} \prod_{i=0}^K b_i^{N_{s-1,i}}- \Theta_0 \prod_{i=0}^K b_i^{N_{t,i}}.$$
Therefore, 

\begin{align*} E | \tilde \Theta_\infty - \tilde \Theta_t| &\le \sum_{s=t+1}^\infty (\rho^{s-1})  + \rho^{t}\\
& = \rho^t (\frac{1}{1-\rho} +1)\\
& = \rho^t \frac{2-\rho}{1-\rho},
\end{align*} 
\end{proof}
\subsection{Convergence of CDFs}
We will make the following  assumptions 
\begin{assum}
\label{assum:regularity} 
Assumption \ref{assum:theta} holds, and 
\begin{enumerate} 
\item For every $z\in [0,1]$ and $i \in \{0,\dots,K\}$, the function $u \to S_i (u) z$ is right-continuous and nondecreasing. 
\item The function $u\to \Theta_0(u)$ right-continuous and non-decreasing on $[0,1)$ and has range contained in $[0,1]$. 
\end{enumerate} 

\end{assum} 
Observe that under Assumption \ref{assum:regularity}, for every $t\in\Z_+$, the function $u\to \Theta_t (u)$ can be extended to a CDF by letting
\begin{equation} 
\label{eq:extension} \Theta_t (u) = \begin{cases} 0 & u < 0 \\   1 & u \ge 1 \end{cases} 
\end{equation} 
\begin{prop}
Let Assumption  \ref{assum:regularity} hold. Then the distribution of the random CDFs $\Theta_t (\cdot)$ converges as $t\to\infty$ to the distribution of the random CDF $\tilde \Theta_\infty$: 
$$ \tilde \Theta_ \infty(u) = \begin{cases} 0 & u < 0; \\ \sum_{s=1}^\infty a_{B_s}(u) \prod_{i=0}^K b_i^{N_{s-1,i}}(u) & u \in [0,1);\\
 1 & u \ge 1. \end{cases}\mbox{ a.s.} $$ 
\end{prop}
\begin{proof}
We extend $\tilde \Theta_t (\cdot)$ to $\R$ analogously to \eqref{eq:extension}. 
As a non-decreasing and right-continuous function  is determined by the values it attains on the rationals, and all finite-dimensional distributions of $\Theta_t(\cdot)$ and of $\tilde \Theta_t(\cdot)$ coincide, it follows the two  function-valued processes $(\Theta_t (\cdot):t \in\Z_+)$ and $(\tilde \Theta_t (\cdot):t\in\Z_+)$ are identically distributed. Since a.s. convergence implies convergence in distribution, it is enough to show that $\tilde \Theta_t (\cdot)$ converges a.s. to the prescribed limit. 

Let $(u_n:n\in\N)$ be a sequence in $[0,1)$ increasing to $1$, and let $A_n$ be the event that the right hand side of \eqref{eq:partial_sum} converges for $u=u_n$. By \eqref{eq:master_chain}, $\Theta_t(u)$ is a (random) composition of the functions  $S_0\dots,S_K$, all of which are increasing in $u$, so we have  $A_{n+1}\subseteq A_n$. Now let $A= \cap_{n=1}^\infty A_n$. Then, $$P(A) = \lim_{n\to\infty} P(A_n) =1.$$ 
Therefore, on $A$, 
$\lim_{t\to\infty} \tilde \Theta_t (u)$ exists for all $u\in [0,1)$, and is equal to $\tilde \Theta_\infty (u)$, the expression on the righthand side of \eqref{eq:thetainfinity_u}. By monotone convergence, this latter expression is right-continuous on $[0,1)$. We have therefore shown that $\tilde \Theta_t (\cdot)$ converges pointwise to $\tilde \Theta_\infty (\cdot)$ on $[0,1)$ a.s. Finally, extend $\tilde \Theta_\infty$ to $\R$ according to \eqref{eq:extension} and the result follows. 
\end{proof} 

\subsection{Proof of Theorem 
\ref{thm:limit}}

We start with the following corollary of Theorem \ref{thm:convergence Theta}.

\begin{cor} 
\label{cor:twofunctions}
Let $K=1, a_0(u)=0$, and  $(p_0,p_1) =(p,1-p)$ for some $p\in (0,1)$. Then,
for $u \in [0,1)$, 
\begin{equation} 
\label{eq:K2_limit} \lim_{t\to\infty}\tilde \Theta_t(u)= \frac{a_1(u)}{b_0(u)} \sum_{k=0}^\infty   b_0^{T_k}(u) (\frac{b_1(u)}{b_0(u)})^{k} \mbox{ a.s.},
\end{equation} 
where $T_k = G_0 +\dots + G_k$, and $G_0,G_1,\dots $ are IID $\mbox{Geom}(1-p)$.
\end{cor} 
\begin{proof} 
Since $a_0(u)=0$, the summation in  the expression for $\tilde \Theta_\infty(u)$ in Theorem \ref{thm:convergence Theta} is  only over those $s$ such that $B_s=1$. Let $T_{-1}=0$ and continue inductively, letting $T_{k} = \inf \{t> T_{k-1}:B_t =1\}$, $k\in \Z$. Then $(T_k-T_{k-1}:k\in \N)$ is an IID sequence with distribution $\mbox{Geom}(1-p)$, or, equivalently, $T_k $ is the partial sum of exactly $k+1$ IID $\mbox{Geom}(1-p)$. Note that $N_{T_{k}-1,1}= k$ and $N_{T_k-1,0}={T_k-1-k}$. Therefore,  

\begin{align*}  \tilde \Theta_\infty(u) &= \sum_{k=0}^\infty a_1(u) b_0(u)^{T_k-1-k} b_1^{k}(u)\\
 & =\frac{a_1(u)}{b_0(u)} \sum_{k=0}^\infty   b_0^{T_k}(u) (\frac{b_1(u)}{b_0(u)})^{k}.
\end{align*} 
\end{proof} 
We now apply Corollary \ref{cor:twofunctions} to the case $$(a_0(u),b_0(u))=(0,u)\mbox{ and } (a_1(u),b_1(u))=(u,1-u).$$ 
This proves \eqref{eq:limit}.

Observe that since the function on the right hand side of \eqref{eq:K2_limit} is continuous a.s. it follows that the limit holds for all $u \in (0,1)$, a.s. This implies that the distribution of the random function $\tilde \Theta_t (\cdot)$ converges as $t\to\infty$ to the distribution of the function on the right hand side of \eqref{eq:K2_limit}, completing the proof of Theorem \ref{thm:limit}. 
\section{Properties of the self-similar measure}

\subsection{Continuity}
\begin{prop} 
\label{prop:CDF}
Let Assumption \ref{assum:theta} hold and assume now that $b_i (u)>0$ for all $i$ and that $S_i(u)$ is $1-1$  for all $i$, and also that for $i\ne j$, the intersection of the images of $S_i(u)$ and $S_j(u)$ is either empty or contains exactly one element.  
 Then the distribution of $\Theta_\infty (u)$ is continuous and its CDF satisfies 
$$ G_u (z) =\frac{ G_u (S_{i'}(u) z) - \sum_{i<i'} p_i}{p_{i'}}$$
for all $i'\in \{0,\dots,K\}$, $ u \in [0,1)$ and $z \in \R$. 
\end{prop}

\begin{proof}
Since each of the Markov chains $(\Theta_t(u):t \in\Z_+)$, $u \in [0,1)$ converges as $t\to\infty$ to a unique distribution indexed by $u$, it follows that the limiting distribution is the unique stationary distribution for the given dynamics. On the other hand, if $\mu_u$ is the stationary distribution for that chain, then 
$$ \mu_u(\cdot)  = \sum_{i=0}^K p_i \mu_u \circ  S^{-1}_i (u)(\cdot) .$$

Hence, for every $x$, 
\begin{equation}
\label{eq:atom}\mu_u (\{x\}) = \sum_{i=0}^Kp_i \mu_u \circ S_i(u)^{-1} (\{x\}),
\end{equation} 
and there can be at most two distinct $i$'s such that $S_i(u)^{-1}(\{x\})$ is not empty.  

We now prove that $\mu_u$ has no atoms. By contradiction, assume that the set of atoms is not empty. Since the cumulative distribution function $G_u$ is increasing and every atom for $\mu_u$ is a discontinuity point for $G_u$ there are at most countably many atoms for $\mu_u$. Since the sum over all atoms $\sum_z\mu_u(\{z\})\leq 1$ it is easy to see that $\mu_u(\{z\})$ attains a maximum for some $z=x$.

Then either 
\begin{enumerate} 
\item  $x$ is in the image of $S_i (u)$ for exactly one $i$. Let $\{z\} = S_i(u)^{-1}(\{x\})$. By \eqref{eq:atom} 
\begin{align*}
\mu_u (\{x\}) =&p_i \mu_u \circ S_i(u)^{-1} (\{x\})\\
=& p_i \mu_u(\{z\}) \\
<& \mu_u (\{x\}),
\end{align*}
where we used the assumed maximality of $\mu_u(\{x\})$. This is  a contradiction; 
\item  or $x$ is the image of $S_i(u)$ and $S_j(u)$ for $i\not =j$. Then,  there is a unique $i\in \{0,\dots,K-1\}$ such that $S_i(u) 1=x=S_{i+1}(u)0$. By
\eqref{eq:atom} $\mu_u(\{x\})$ is equal to  $p_i \mu_u(\{1\})+p_{i+1}\mu_u(\{0\})$. By the maximality of $\mu_u(\{x\})$, it follows that $\mu_u(\{0\}) = \mu_u(\{1\}) = \mu_u (\{x\})$ and $p_i+p_{i+1}=1$. Therefore, we necessarily have $K=1$ and $i=0$. We are now back to case 1 for $x=0$ and $x=1$. That is, $\mu_u$ reaches a maximum at $0$ and at $1$ but $0$ and $1$
are each the image of a single $S_i(u)$. As in case 1 above this leads to a contradiction. 
\end{enumerate}

Therefore $\mu_u$ has no atoms and the corresponding distribution function is continuous.

We now turn to the proof of the  second statement in Proposition \ref{prop:CDF}.

 If $I_i$ is the image of $[0,1]$ under $S_i(u)$,  then $\mu_u(I_i)=p_i$, and if $G_u$ is the CDF of $\mu_u$, 

$$ G_u (x) = \sum_{i=0}^K p_i G_u (S^{-1}_i (u) (x)).$$ 

If $x\in I_{i'}$, the sum above becomes 

$$ G_u (x) = \sum_{i<i'} p_i +p_{i'} G_u (S^{-1}_{i'}(u) (x)) +0.$$ 

We can rewrite this by letting $z =S^{-1}_{i'}(u)(x)=(x - a_{i'}(u))/b_i(u)$ to obtain the following,
$$ G_u (z) =\frac{ G_u (S_{i'}(u) z) - \sum_{i<i'} p_i}{p_{i'}}.$$
\end{proof}
%We are ready to prove Theorem \ref{th:limit} 
%\begin{example} 
%If  $S_0(x)= ux$ and $S_1(x) = u+ (1-u) x$ with $(p_0,p_1)= (p,1-p)$, then 
%$$ \tilde \Theta_\infty  = \sum_{k=0}^\infty u^{T_k}(\frac{1-u}{u})^{k}.$$ 
%\end{example} 

\subsection{Local exponents and singularity} 
Next, we will get some additional information on the behavior of $\mu_u$ in the particular case $K=1$, $S_0(x) = u+ (1-u)x$,  $S_1 (x) = ux$ and $(p_0,p_1)=(1-p,p)$. We already know that $\mu_u$ has no atoms because $G_u$ is continuous. We will next explore its local behavior, and will use that to conclude that with the exception of the case $p=u$, $\mu_u$ is singular with respect to the Lebesgue measure. 

Recall that $D$ denotes the set of all points in the form \eqref{eq:the_y}. Then $D$ is countable and dense, and therefore $\mu_u(D)=0$. We will show that the local exponent of $\mu_u$ on $D$ is quite different from the local exponent outside of $D$, see Figure \ref{fig:2} 

\subsubsection{Local exponent on D}
We will first look at the behavior of $G_u$ near points in $D$. Let $y\in D$
\begin{equation} 
\label{eq:look_at_y}y = \sum_{l=1}^m u^{n_l}(1-u)^{l-1}
\end{equation} 
be of the form $m\ge 2$ and $n_1,\dots,n_m$ all finite. We make this restriction only to simplify the argument.
Now, let   $\delta = u^{n_m+k}(1-u)^{m}$. Then, by 
\eqref{eq:Gu_J},
\begin{align}  
\nonumber G_u (y+\delta) -G_u (y) & =  p^{n_m+k}(1-p)^{m} \\
\nonumber
 &= u^{\frac {\ln p}{\ln u} (n_m+k)}(1-p)^{m}\\
 \nonumber 
 &= (\delta / (1-u)^{m})^{\frac{\ln p }{\ln u}} (1-p)^{m}\\
 & = C(y,u,p) \delta^{\frac{\ln p}{\ln u}}.
\label{eq:right_behave} 
\end{align} 
We will now approach the same $y$ from the left. 
\begin{align*} 
y &= \sum_{l=1}^m u^{n_l} (1-u)^{l-1} \\
 & = \sum_{l=1}^{m-1} u^{n_l}(1-u)^{l-1} + u^{n_m}(1-u)^{m-1}\\
 & = \sum_{l=1}^{m-1} u^{n_l}(1-u)^{l-1} + \sum_{l'=m}^\infty u^{n_{m}+1}(1-u)^{l'-1}.
 \end{align*} 

For $k\in \N$, let
 $$\delta = \sum_{l'=m+k}^\infty u^{n_{m}+1}(1-u)^{l'-1}
= u^{n_m}(1-u)^{m+k-1}.$$ 
Note that as $k\to\infty$, $\delta \to 0$. With this choice, 
$$y-\delta =\sum_{l=1}^{m-1} u^{n_l}(1-u)^{l-1} + \sum_{l'=m}^{m+k-1} u^{n_{m}+1}(1-u)^{l'-1}$$ 
and so the by the continuity of $G_u$ we have 
\begin{align}
\nonumber G_u(y) - G_u(y-\delta) &= \sum_{l'=m+k}^\infty p^{n_m+1}(1-p)^{l'-1}\\
\nonumber
& =p^{n_m}(1-p)^{m+k-1} \\
\nonumber
& =p^{n_m}( 1-u)^{(m+k-1) \frac{\ln (1-p)}{\ln (1-u)} }\\
\nonumber
& =p^{n_m}(\delta / u^{n_m})^{\frac{\ln (1-p)}{\ln (1-u)}}\\
\label{eq:left_behave}
& =C'(y,u,p) \delta^{\frac{\ln (1-p)}{\ln (1-u)}}.
\end{align}
Combining \eqref{eq:right_behave} and \eqref{eq:left_behave} we obtain that 
\begin{align}
\label{eq:Dright}
\lim_{\delta \to 0+} & \frac{ \ln | G_u (y+\delta) - G_y(y)|}{\ln \delta} =  \frac{\ln p}{\ln u}\\
\label{eq:Dleft}
\lim_{\delta \to 0-} & \frac{ \ln | G_u (y+\delta) - G_y(y)|}{\ln \delta} =  \frac{\ln (1-p)}{\ln (1-u)}
\end{align}
Note that when the right hand side in \eqref{eq:Dright} or \eqref{eq:Dleft} is larger or equal to  $1$, then the respective one-sided derivative exists, and is equal to zero if the limit is strictly larger than $1$. Furthermore, $\ln p < \ln u $ if and only if $\ln (1-p) > \ln (1-u)$, therefore exactly one of the one-sided limits is larger than $1$ except for $u=p$.
\subsubsection{Local exponent outside of D}
 To ease notation we will freeze $u$ and drop the dependence on it. For each $ t \in \N$, let $\epsilon  = (\epsilon_1,\dots,\epsilon_t) \in \{0,1\}^{t}$.  Write $|\epsilon| = \sum \epsilon_i$. 
Let 
$$I(\epsilon) = S_{\epsilon_1}\circ S_{\epsilon_{t-1}} \cdots \circ S_{\epsilon_t}\left( [0,1]\right).$$ 
Note that we compose from the right rather than the left. Let  ${\cal F}_t=\{I(\epsilon):\epsilon \in \{0,1\}^t\}$. Since $[0,1]= S_0 ([0,1])\cup S_1([0,1])$, it follows from induction  that the union of all elements in ${\cal F}_t$ is $[0,1]$. In what follows, if $I$ is an interval, write $|I|$ for its length. If $\epsilon \in \{0,1\}^t$, then  $|I(\epsilon)|= (1-u)^{t-|\epsilon|}u^{|\epsilon|}$. Now 
$$\sum_{\epsilon\in \{0,1\}^t} |I(\epsilon)| = \sum_{k=0}^t \binom{t}{k}(1-u)^{t-k}u^{k} = 1,$$ 
Therefore the intervals in ${\cal F}_t$  are non overlapping. We also observe that the points which belong to exactly two intervals are elements in the countable set $D$. 

Suppose $x \in [0,1]\cap D^c$. Then for every $t$ there exists a unique element $I_t (x) \in {\cal F}_t$ such that $x \in I_t(x)$. Also since every element in ${\cal F}_{t+1}$ is a subset of a unique element in ${\cal F}_{t}$, it follows that $I_{t+1}(x) \subset I_t (x)$. As a result, $$\bigcap_{t=1}^\infty I_t (x) = \{x\},$$ 
and there exists a unique sequence  $\epsilon\in \{0,1\}^{\N}$ such that $x \in I(\epsilon_1,\dots,\epsilon_t)$ for all $t$. Note that $\epsilon_1,\epsilon_2,\dots$ are all functions of $x$, and that for every $I \in {\cal F}_t$, $x\in I$ if and only if $I = I(\epsilon_1(x),\epsilon_2(x),\dots,\epsilon_t(x))$. 

Let's sample $X$ according to $\mu$ and fix $(\epsilon'_1,\epsilon'_2,\dots,\epsilon'_t)\in \{0,1\}^t$. Then 
\begin{align*}  P( X \in I(\epsilon'_1,\epsilon'_2,\dots,\epsilon'_t) ) &= p_0 \mu \left(x: S_0x \in  S_{\epsilon'_2} \dots\circ S_{\epsilon'_t}([0,1])\right) \\
&+ p_1 \mu \left(x:S_1 x \in  S_{\epsilon'_2} \dots\circ S_{\epsilon'_t}([0,1])\right) \\
& = p_{\epsilon'_1}\mu (x: x\in S_{\epsilon'_2}\circ\dots\circ S_{\epsilon'_t}([0,1])).
\end{align*}
Iterating,  
$$ P( X \in I(\epsilon'_1,\dots,\epsilon'_t)) = \prod_{i=1}^t p_{\epsilon'_i}.$$
Equivalently, 
$$ P( \epsilon_1(X)=\epsilon'_1,\epsilon_2(X) = \epsilon'_2,\dots,\epsilon_t(X) = \epsilon'_t) = \prod_{i=1}^t p_{\epsilon'_i}.$$ 
That is, when $X$ is sampled according to $\mu$, the RVs $(\epsilon_1(X),\epsilon_2(X),\dots)$ are IID $\mbox{Bern}(p_1)$. Let $B_t(x) = \sum_{i=1}^t \epsilon_i(x)$ and recall that $I_t(x)$ is the unique interval in ${\cal F}_t$ containing $x$. Then by construction  $\mu(I_t (x)) = (1-p)^{t-B_t(x))} p^{B_t (x)}$, while $|I_t(x)|=(1-u)^{t-B_t(x))} u^{B_t (x)}$. 
Therefore by the law of large numbers
\begin{equation} 
\label{eq:asdim} \lim_{t\to\infty} \frac{ \ln \mu (I_t(x)) }{\ln |I_t(x)|}= \frac{p \ln p + (1-p)\ln (1-p)}{p \ln u + (1-p) \ln (1-u)},\quad \mu\mbox{ a.s.}
\end{equation} 
\begin{prop} 
\label{prop:Sing}
For every $u\not =p$ the measure $\mu_u$ is singular
with respect to the Lebesgue measure.
\end{prop}

\begin{proof}
The r.h.s of \eqref{eq:asdim} is in  $(0,1]$ and is equal to $1$ if and only if $u=p$, in which case $\mu_u$ is uniform. Assume now $u\not =p$. Let $N$ be the differentiability set of $G_u$ in $(0,1)$.  Since $G_u$ is nondecreasing, $N$ has Lebesgue measure one.  However, \eqref{eq:asdim} shows that $\mu_u (N)=0$. Therefore $\mu_u$ is singular with respect to the Lebesgue measure. 
\end{proof}
\begin{figure}
\label{fig:2} 
\begin{tikzpicture}
\begin{axis}[
    axis lines = left,
    xlabel = $u$,
    ylabel = {Local Exponent},
]
%Below the red parabola is defined
\addplot [
    domain=0.001:0.4, 
    samples=500, 
    color=red
]
{min(ln(0.4)/ln(x),1)};
%\addlegendentry{\eqref{eq:Dleft}}
\addplot [
    domain=0.4:.999, 
    samples=500, 
    color=red,
    %dotted
]
{min(ln(0.6)/ln(1-x),1)};
%Here the blue parabloa is 
%\addlegendentry{\eqref{eq:Dright}}
\addplot [
    domain=0:1, 
    samples=500, 
    color=blue,
    ]
    {(0.4*ln(0.4)+0.6*ln(0.6))/(0.4*ln (x)+ 0.6*ln(1-x))};
%\addlegendentry{\eqref{eq:asdim}}
\end{axis}
\end{tikzpicture}
\caption{The local exponents for $\mu_u$ as a function of $u$, with $p=0.4$. The red graph represents the minimum of the left and right limits from \eqref{eq:Dleft} and \eqref{eq:Dright}. The blue graph represents the a.s. limit of \eqref{eq:asdim}. }
%The local exponents for $\mu_u$ as a function of $u$, with $p=0.4$. The solid and dotted red lines represent the left and right  limits from \eqref{eq:Dleft} and \eqref{eq:Dright}, respectively, with each restricted to the range of $u$  where the respective one-side limit is less than $1$. The blue line represents the a.s. limit of \eqref{eq:asdim}} 
\end{figure}
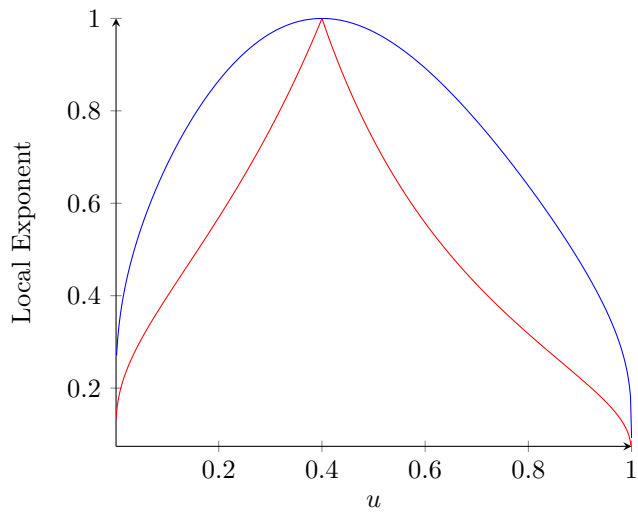
~~~~~
\bibliographystyle{amsplain}

\end{document}